\documentclass[11pt]{article}
\usepackage{amsmath, amssymb, mathtools, palatino, euler, epic, eepic, floatflt, microtype}
\usepackage{fullpage}
\usepackage{amscd,amsthm,tikz}
\usetikzlibrary{intersections}
\usepackage[margin=1 in]{geometry}
\definecolor{darkblue}{rgb}{0,0,0.7}
\definecolor{darkred}{rgb}{0.7,0,0}
\usepackage[colorlinks,linkcolor=darkblue, citecolor=darkblue,
pagebackref=true,pdftex]{hyperref}
\usepackage{cleveref}

\tikzset{help lines/.style={dashed, thick}}
\usepackage[T1]{fontenc}
\usepackage[ansinew]{inputenc}
\usepackage{float}
\usepackage{setspace}
\usepackage{xpatch}
\usepackage{subcaption}
\usepackage{pst-node}
\usepackage{tikz-cd}
\usetikzlibrary{positioning}
\usetikzlibrary{shapes,arrows}
\usepackage{tikz}
\usepackage{graphicx,epsfig}
\usepackage{mathrsfs}

\newtheorem{theorem}{Theorem}[section]
\newtheorem{lemma}[theorem]{Lemma}
\newtheorem{corollary}[theorem]{Corollary}
\newtheorem{definition}{Definition}

\newtheorem{proposition}[theorem]{Proposition}
\newtheorem{claim}{Claim}

\def\S{\mathbb{S}}

\def\F{\mathcal{F}}
\def\bd{\mathrm{BD}}
\def\NC{\mathcal{N}\mathcal{C}}
\def\vl{\vec\lambda}
\def\lk{\mathrm{lk}}
\def\del{\mathrm{del}}
\def\ind{\mathrm{Ind}}
\newcommand{\si}{\sigma}
\newcommand{\sm}{\setminus}
\numberwithin{equation}{section}

 \title{Vertex decomposability of complexes associated to forests}
\date{}

\begin{document}

 \maketitle
 \vspace{-2cm}
\begin{center}
{ { {\large  Anurag Singh}\footnote{The author is partially funded by a grant from Infosys Foundation.}}\\
{\footnotesize Chennai Mathematical Institute, India}\\

{\footnotesize e-mail: {\it anuragsingh@cmi.ac.in}}}\\
\end{center}
\medskip
\vspace{-0.25in}
\begin{abstract}

In this article, we discuss the vertex decomposability of three well-studied simplicial complexes associated to forests. In particular, we show that the bounded degree complex of a forest and the complex of directed trees of a multidiforest are vertex decomposable. We then prove that the non-cover complex of a forest is either contractible or homotopy equivalent to a sphere. Finally we provide a complete characterization of forests whose non-cover complexes are vertex decomposable. 

\end{abstract}
\vspace{0.1cm}
\hrule
%\vspace{.1in}

%%%%%%%%%%%%%%%%%%%%%%%%=========================

\section{{Introduction}}
The concept of a (pure) vertex decomposable simplicial complex was introduced by Provan and Billera \cite{PB80} in order to study the diameter problems. Later, in \cite{BW97}, Bj{\"o}rner and Wachs extended this notion to non-pure  simplicial complexes. Defined in a recursive way (see \Cref{def:vertexdecomposable}), the notion of vertex decomposability enjoys a very rich literature. In \cite{wachs99}, Wachs showed that a vertex decomposable simplicial complex is shellable (see \Cref{def:shellability}) and hence sequentially Cohen-Macaulay\footnote{See \cite[Definition 3.22]{Jon08} for the definition of sequentially Cohen-Macaulay complex.} \cite[Theorem 3.33]{Jon08}. In this direction, we have the following strict implications:
\begin{equation}
    \text{vertex decomposable} \Longrightarrow \text{shellable}
\Longrightarrow \text{sequentially Cohen-Macaulay}
\end{equation}

Let $K$ be a simplicial complex on vertex set $\{x_1,\dots,x_n\}$ and let $R=\mathbb{F}[x_1,\dots,x_n]$ denotes the polynomial ring on $n$ variables over some field $\mathbb{F}$. The monomial ideal associated to $K$, denoted $I_{K}$, is the ideal in $R$ generated by all monomials $x_{i_1}\dots x_{i_s}$ whenever $\{x_{i_1},\dots ,x_{i_s}\}\notin K$. The \emph{Stanley-Reisner ring} of $K$ is the quotient ring $\mathbb{F}[K]:= R/I_{K}$. It is known (see for example \cite{BWW09}) that a complex $K$ is sequentially Cohen-Macaulay if and only if $\mathbb{F}[K]$ is sequentially Cohen-Macaulay in the algebraic sense (see \cite[Definition III 2.9]{stan07} for the definition of the later). This connection between algebra and topology motivated researchers in the last decade to explore the vertex decomposability of a complex $K$ in order to study the algebraic peoperties of $\mathbb{F}[K]$ (see for instance \cite{ED09, wood09, Mor19, CR17}).

%For more general results in this direction, the interested reader is referred to \cite[Section 12]{BW97} and \cite{FMS14}.

Note that a vertex decomposable simplicial complex is homotopy equivalent to a wedge of spheres since it is shellable \cite[Theorem 4.1]{BW96}. However, the converse is not true in general (for example the disjoint union of two simplices of dimension $1$ is homotopy equivalent to a $0$-sphere but not shellable and hence not vertex decomposable). Also the boundary complexes of simplicial polytopes are shellable \cite[Chapter 8]{ziegler12} but many of them are not vertex decomposable \cite[Section 6]{klee87}. Recently, in \cite{CDGO20}, Coleman et al. showed that a vertex decomposable complex is shelling completable (see \cite[Definition 1.2]{CDGO20}) and hence satisfy the Simon's Conjecture \cite{simon94}.
%For more results related to vertex decomposable complexes see \cite{GW15, CDGO20, CR17}. 

% Let $G$ be a simple (undirected) graph on vertex set $[n]= \{1,\dots,n\}$ and let $R=\mathbb{F}[x_1,\dots,x_n]$ denotes the polynomial ring on $n$ variables over some field $\mathbb{F}$. The \emph{edge ideal} of $G$, denoted $I_{G}$, is the ideal in $R$ generated by all monomials $x_{i}x_{j}$ whenever $(i,j)$ is an edge in $G$. The \emph{Stanley-Reisner ring} of $I_G$ is the quotient ring $R_G:= R/I_{G}$. Engstr{\"o}m and Dochtermann \cite{ED09} and Woodroofe \cite{wood09} used the vertex decomposability of simplicial complexes in order to study the algebraic properties of $I_G$ and $R_G$. 

%  

 Inspired by the importance of vertex decomposable complexes, in this article we study the vertex decomposability of various simplicial complexes associated to forests. The article is organized as follows. In \Cref{sec:prelim} we recall all the  important  definitions  and relevant tools. In \Cref{sec:boundeddegree}, we prove that the bounded degree complex (a generalization of matching complexes) of a forest is vertex decomposable (cf. \Cref{thm:mainbdcomplex}). \Cref{sec:noncover} is devoted towards the study of non-cover complexes of graphs. Here, we show that the non-cover complex of a forest is either contractible or homotopy equivalent to a sphere. We also give a complete list of forests whose non-cover complexes are vertex decomposable (cf. \Cref{theorem:mainnoncover2}). In the final section we show that the complex of directed trees of a multidiforest is vertex decomposable (cf. \Cref{thm:maindirectedtree}).

\section{Preliminaries}\label{sec:prelim}
An {\itshape(undirected) graph} is an ordered pair $G=(V(G),E(G))$ where $V(G)$ is called the set of vertices and $E(G) \subseteq V(G) \times V(G)$, the set of (unordered) edges of $G$. The graph $G$ is called \emph{simple} if $(v,v)\notin E$ for any $v\in V(G)$. The vertices $v_1, v_2 \in V$ are said to be adjacent, if $(v_1,v_2)\in E$. A vertex $v$ is said to be {\itshape adjacent} to an edge $e$ (and vice versa), if $v$ is an end point of $e$, {\itshape i.e.}, $e=(v, w)$. Two edges $e,f \in E(G)$ are said to be adjacent if both are adjacent to a common vertex. The number of vertices adjacent to a vertex $v$ in $G$ is called the {\itshape degree} of $v$ in $G$, denoted deg$_G(v)$. If deg$_G(v)=1$, then $v$ is called a {\itshape leaf vertex} of $G$ and the edge adjacent to $v$ is called a \emph{leaf edge} of $G$.  Two graphs $G$ and $H$ are called {\itshape isomorphic}, denoted $G\cong H$, if there exists a bijection, $f : V(G) \to V(H)$ such that $(v,w) \in E(G)$ if and only if $(f(v),f(w)) \in E(H).$ 

A graph $H$ with $V(H) \subseteq V(G)$ and $E(H) \subseteq E(G)$ is called a {\it subgraph} of the graph $G$. For a nonempty subset $H$ of $E(G)$, the induced subgraph $G[H]$, is the subgraph of $G$ with edges $E(G[H]) = H$ and $V(G[H]) = \{a \in  V(G)  :  a \text{ is adjacent $e$ for some }e \in H\}$. For a nonempty subset $U$ of $V(G)$, the induced subgraph $G[U]$, is the subgraph of $G$ with vertices $V(G[U]) = U$ and $E(G[U]) = \{(a, b) \in E(G) \ | \ a, b \in U\}$.  For a subset $S \subseteq E(G)$, the induced subgraph $G[S]$, is the subgraph of $G$ with vertices $V(G[S])=\{u \in V(G) : u \text{ is adjacent } e \text{ for some } e \in S\}$ and $E(G[S]) = S$. In this article, $G[V(G)\setminus A]$ will be denoted by $G-A$ for $A\subsetneq V(G)$. For $U\subsetneq V(G)$ and $S\subseteq E(G)$, the graphs $G[V(G)\setminus U]$ and $G[E(G)\setminus S]$ will be denoted by $G-U$ and $G-S$ respectively.

% For $n \geq 1$, {\itshape path graph} of length $n$, denoted $P_n$, is a graph with vertex set $V(P_n) = \{1, \ldots, n\}$ and  edge set $E(P_n) = \{(i,i+1) \ |  \ 1 \leq i \leq n-1\}$. 
A {\itshape tree} is a graph in which any two vertices are connected by exactly one path and a \emph{forest} is a family of disjoint trees. Let $T$ be a tree. Vertex $v\in V(T)$ is called an {\itshape internal vertex} if deg$_T(v)>1$. An internal vertex is called a {\itshape corner vertex}, if it is adjacent to at most one internal vertex. For example: in \Cref{fig:figure of a tree}, $v_2$ is an internal vertex but not a corner vertex, and $v_1,v_3$ are corner vertices.

\begin{figure}[H]
		\centering
		\begin{tikzpicture}
 [scale=0.4, vertices/.style={draw, fill=black, circle, inner sep=1.0pt}]
        \node[vertices, label=below:{$v_1$}] (v1) at (0,0)  {};
		\node[vertices, label=below:{$v_2$}] (v2) at (7,0)  {};
		\node[vertices, label=below:{$v_3$}] (v3) at (14,0)  {};
		\node[vertices, label=above:{}] (w12) at (-3,3)  {};
		\node[vertices, label=above:{}] (w32) at (11,3)  {};
		\node[vertices, label=above:{}] (w11) at (0,4)  {};
		\node[vertices, label=above:{}] (w31) at (14,4)  {};
		
%\foreach \to/\from in {v1/v2};
%\draw [-] (\to)--(\from);
\path (v1) edge node[pos=0.5,below] {} (v2);
\path (v1) edge node[pos=0.5,left] {} (w11);
\path (v2) edge node[pos=0.5,left] {} (v3);
\path (v3) edge node[pos=0.5,left] {} (w31);
\path (v1) edge node[pos=0.5,left] {} (w12);
\path (v3) edge node[pos=0.5,left] {} (w32);
\end{tikzpicture}\caption{}\label{fig:figure of a tree}
\end{figure}

% Let $T$ be a tree, $\vec{\lambda}=(\lambda_1,\dots,\lambda_n)$ be a sequence of non-negative integers and $\vec{v}=(v_1,\dots,v_n)$ denote the sequence of all internal vertices of $T$. Recall that, BD$^{\vec{\lambda}}(T(\vec{v}))$ is a simplicial complex whose faces are all those subgraphs of $T$ in which degree of $v_i$ is at most $\lambda_i$ and degree of any leaf of $T$ is at most $1$. Trees with exactly $1$ or $2$ vertices are trivial trees. Hence, unless otherwise mentioned, we assume that all the trees are finite and have more than 2 vertices.
%The following result explains the case in which permuting internal vertices does not affect bounded degree complex.

\begin{proposition}\label{prop:basic properties of trees}
 Every tree with more than one edge has a corner vertex. 
\end{proposition}
\begin{proof}
We prove this by labelling all the vertices of $T$. Start with a leaf $v$ and label it $1$. Next we give label $2$ to all the vertices adjacent to $v$. We then give label $3$ to all those vertices which are adjacent to at least one vertex labelled $2$ and is not already labelled. We continue labelling vertices of $T$ with this argument. Since $T$ is finite and has no cycle, this labelling stops at certain stage say at $\ell$. Since $|V(T)|\geq 3$, $\ell \geq 3$. Therefore, observe that, any vertex labelled $\ell-1$ will be a corner vertex. 
\end{proof}

The following graphs are a special class of trees.

\begin{definition}\label{def:caterpillar graph}
\normalfont A {\itshape caterpillar graph} is a tree in which every vertex is on a central path or only one edge away from the path (see \Cref{fig:caterpillar graph} for examples). 
\end{definition}

A caterpillar graph of length $n$ is denoted by $G_n(m_1,\dots,m_n)$, where $n$ represents the number of vertices of the central path and $m_i$ denote the number of leaves adjacent to $i^{\mathrm{th}}$ vertex of the central path.

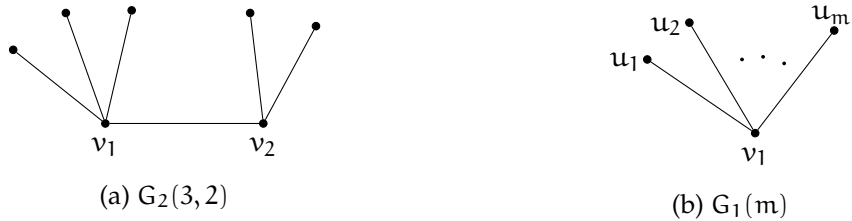
\begin{figure}[H]
	\begin{subfigure}[]{0.55 \textwidth}
		\centering
		\begin{tikzpicture}
 [scale=0.35, vertices/.style={draw, fill=black, circle, inner sep=1.0pt}]
        \node[vertices, label=below:{$v_1$}] (v1) at (0,0)  {};
		\node[vertices, label=below:{$v_2$}] (v2) at (6,0)  {};
		\node[vertices] (l11) at (-3.5,2.8)  {};
		\node[vertices] (l12) at (-1.5,4.2)  {};
		\node[vertices] (l13) at (1,4.3)  {};
		\node[vertices] (l21) at (5.5,4.2)  {};
		\node[vertices] (l22) at (8,3.7)  {};
		
\foreach \to/\from in {v1/v2}
%\draw [-] (\to)--(\from);
\path (v1) edge node[pos=0.5,below] {} (v2);
\path (v1) edge node[pos=0.65,below] {} (l11);
\path (v1) edge node[pos=0.5,left] {} (l12);
\path (v1) edge node[pos=0.5,left] {} (l13);
\path (v2) edge node[pos=0.5,left] {} (l21);
\path (v2) edge node[pos=0.5,left] {} (l22);
\end{tikzpicture}\caption{$G_2(3,2)$}\label{fig:G243}
	\end{subfigure}
	\begin{subfigure}[]{0.35 \textwidth}
		\centering
	\begin{tikzpicture}
 [scale=0.35, vertices/.style={draw, fill=black, circle, inner sep=1.0pt}]
        \node[vertices, label=below:{$v_1$}] (v1) at (0,0)  {};
		\node[vertices] (l11) at (-4.1,2.8)  {};
		\node[vertices] (l12) at (-2.5,4.2)  {};
		\node[vertices] (l1m) at (3.0,3.9)  {};
		\node[vertices,inner sep=0.3pt] (d1) at (-0.5,2.8)  {};
		\node[vertices,inner sep=0.3pt] (d2) at (0.3,2.9)  {};
		\node[vertices,inner sep=0.3pt] (d3) at (1.1,2.65)  {};
		
\foreach \to/\from in {v1/l11}
%\draw [-] (\to)--(\from);
\path (v1) edge node[pos=1,left] {$u_1$} (l11);
\path (v1) edge node[pos=1,left] {$u_2$} (l12);
%\path (v1) edge node[pos=0.5,left] {$l_1^3$} (l13);
\path (v1) edge node[pos=1,above] {$u_m$} (l1m);
\end{tikzpicture}\caption{$G_1(m)$}\label{fig:G1m}
	\end{subfigure}
	\caption{Caterpillar graphs} \label{fig:caterpillar graph}
\end{figure}

An {\itshape (abstract) simplicial complex} $K$ is a collection of finite sets such that if $\sigma \in K$ and $\tau \subseteq \sigma$, then $\tau \in K$. The elements  of $K$ are called {\itshape simplices} of $K$.  If $\sigma \in K$ and $|\sigma |=k+1$, then $\sigma$ is said to be {\it $k$-dimensional}. The \emph{dimension} of $K$, denoted dim$(K)$, is the maximum of the dimensions of its simplices. Further, if $\si \in K$ and $\tau \subseteq \si$ then $\tau$ is called a {\itshape face} of $\si$ and if $\tau \neq \si$ then $\tau$ is called a {\itshape proper face} of $\si$. The set of $0$-dimensional simplices of $K$ is denoted by $V(K)$, and its elements are called {\it vertices} of $K$. Maximal simplices of $K$ are called \emph{facets} and $K$ is called \emph{pure} if all its facets have the same dimension. 

A {\it subcomplex} of a simplicial complex $K$ is a simplicial complex whose simplices are contained in $K$. For $\si \in K$, the boundary of simplex $\si$, denoted $\partial(\si)$ is collection of all proper faces of $\si$. If $\si$ is a vertex, then $\partial(\si)=\emptyset$. For $k\geq 0$, the {\itshape $k$-skeleton} of a simplicial complex $K$ is the collection of all those simplices of $K$ whose dimension is at most $k$.

For a simplex $\si \in K$ , define
\begin{equation*}
    \begin{split}
        \mathrm{lk}(\sigma,K) & := \{\tau \in K : \sigma \cap \tau = \emptyset,~ \sigma \cup \tau \in K\}, \\
        \mathrm{del}(\sigma,K) & := \{\tau \in K : \sigma \nsubseteq \tau\}.
    \end{split}
\end{equation*}
The simplicial complexes $\mathrm{lk}(\sigma,K)$ and $\del(\sigma,K)$ are called \emph{link} of $\si$ in K and \emph{(face) deletion} of $\si$ in $K$ respectively. The join of two simplicial complexes $K_1$ and $K_2$ , denoted as $K_1 \ast K_2$, is a simplicial complex whose simplices are disjoint union of simplices of $K_1$ and of $K_2$. Let $\Delta^S$ denotes a $(|S|-1)$-dimensional simplex with vertex set $S$ and $\partial(\Delta^S)$ denotes the boundary of simplex $\Delta^S$, {\itshape i.e.}, $(|S|-2)-$skeleton of $\Delta^S$. Then, the \emph{cone} on $K$ with apex $a$, denoted as $C_a(K)$, is defined as
 $$ C_a (K ) := K \ast \Delta^{\{a\}}.$$
 
 For $a , b \notin V (K )$, the suspension of $K$, denoted as $\Sigma(K )$, is defined as
$$\Sigma (K ) := K \ast \partial(\Delta^{\{a,b\}}).$$

Observe that, for any vertex $v \in V (K )$, we have
$$ K = C_v (\lk(v,K)) \cup \del(v,K) \text{ and } C_v (\lk(v,K)) \cap \del(v,K) = \lk(v,K).$$
Clearly, $C_v (\lk(v,K))$ is contractible. Therefore, if $\lk(v,K)$ is contractible in $\del(v,K)$ then from \cite[Remark 2]{SSA20} we get the following homotopy equivalence 

\begin{equation}\label{eq:linkdel}
    K \simeq \del(v,K)\bigvee \Sigma (\lk(v,K))
\end{equation}
 here $\bigvee$ denotes the wedge of topological spaces. In this article, empty wedge will mean that the space is contractible.
 
\begin{definition}\label{def:vertexdecomposable}
\normalfont A simplicial complex $K$ is called \emph{vertex decomposable} if $K$ is
a simplex, or $K$ contains a vertex $v$ such that
\begin{enumerate}
    \item[(i).] both $\lk(v,K)$ and $\del(v,K)$ are vertex decomposable, and
    \item[(ii).] any facet of $\del(v,K)$ is a facet of $K$.
\end{enumerate}
A vertex $v$ which satisfies condition (ii) is called a \emph{shedding vertex} of $K$. We call vertex $v$ a \emph{decomposing vertex} of $K$ if $v$ satisfies both the conditions (i) and (ii).
\end{definition}

The following can be easily inferred from \Cref{def:vertexdecomposable}.
\begin{proposition}\label{prop:join of vd}
For two simplicial complexes $K_1$ and $K_2$, the join $K_1\ast K_2$ is vertex decomposable if and only if both $K_1$ and $K_2$ are vertex decomposable.
\end{proposition}

\begin{definition}\label{def:shellability}
\normalfont A simplicial complex $K$ is called \emph{shellable} if its facets can be arranged in linear order $F_1, F_2, \dots ,F_t$ in such a way that the subcomplex $\big(\bigcup\limits_{1\leq j <r}\Delta^{F_j}\big) \cap \Delta^{F_r}$ is pure and
$(\text{dim}(\Delta^{F_k}) -1)$-dimensional for all $k = 2,\dots,t$. Such an ordering of facets is called a
shelling order of $K$.
\end{definition}

\begin{lemma}[ {\cite[Lemma 3.3]{bar13}} ]\label{lemma:glueing}
Let $K_1$ and $K_2$ be two contractible subcomplexes of a simplicial complex $K$ such that $K = K_1 \cup K_2$. Then $K \simeq \Sigma(K 1 \cap K 2 )$, where $\Sigma(X)$
denotes the suspension of space $X$.
\end{lemma}
 
For a space $X$, let ${\Sigma}^r(X)$ denote its $r$-fold suspension, where $r\geq 1$ is a natural number. Recall that, there is a homotopy equivalence
 \begin{equation}\label{remark:join of space with spheres is suspension}  
 \begin{split}
     \S^{r-1}  \ast X & \simeq \Sigma^{r}(X), \text{ and }\\
     \Sigma^{r}(\S^{t}) & \simeq \S^{t+r+1}.
 \end{split}
\end{equation}

% \section{Complexes associated to forests}
% In this section, we discuss various simplicial complexes associated to forests and their vertex decomposability. Throughout this section, except in \Cref{subsec:directedtrees}, all the graphs will be simple unless otherwise explicitly mentioned.

\section{The bounded degree complex}\label{sec:boundeddegree}

Let $G$ be a graph and $\vl: V(G)\longrightarrow \mathbb{Z}_{\geq 0}$ be a labelling of the vertices of $G$ with non-negative integers. The \emph{bounded degree complex}, denoted $\text{BD}^{\vec{\lambda}}(G)$, is a simplicial complex whose vertices are the edges of $G$ and faces are subsets $\sigma\subseteq E(G)$ such that for each $v\in V(G)$, the degree of vertex $v$ in the induced subgraph $G[H]$ is at most $\vl(v)$ (see \Cref{fig:example of BD} for example). When $\vl(v) =k$ for all $v\in V(G)$, the bounded degree complex $\text{BD}^{\vl}(G)$ is called the {\itshape $k$-matching complex} of graph $G$ and denoted by $M_k(G)$. 

% {\bf Example:}
% See \Cref{fig:example of BD} consists of a graph $G$ on $5$ vertices and BD$^{\vec{\lambda}}(G)$ for $\vec{\lambda}=(2,1,1,1,1)$. The complex BD$^{\vec{\lambda}}(G)$ consists of $3$ maximal simplices, namely $\{e_1,\ell_1^1\}, \{e_1,\ell_1^2\}$ and $\{\ell_1^1,\ell_1^2,\ell_2^1\}$.

\begin{figure}[H]
	\begin{subfigure}[]{0.5 \textwidth}
		\centering
		\vspace{0.5cm}
		\begin{tikzpicture}
 [scale=0.5, vertices/.style={draw, fill=black, circle, inner sep=1.0pt}]
        \node[vertices, label=below:{}] (v1) at (0,0)  {};
		\node[vertices, label=below:{}] (v2) at (6,0)  {};
		\node[vertices, label=above:{}] (l11) at (-2.5,3.5)  {};
		\node[vertices, label=above:{}] (l12) at (0.7,3.8)  {};
		\node[vertices, label=above:{}] (l21) at (5,3.7)  {};
		\node[vertices, label=above:{}] (l22) at (7.5,3.9)  {};
		\node[vertices, label=above:{}] (v3) at (-4,0)  {};
		
\foreach \to/\from in {v1/v2}
%\draw [-] (\to)--(\from);
\path (v1) edge node[pos=0.5,below] {$e_6$} (v2);
\path (v1) edge node[pos=0.45,left] {$e_2$} (l11);
\path (v1) edge node[pos=0.5,right] {$e_3$} (l12);
\path (v2) edge node[pos=0.5,left] {$e_4$} (l21);
\path (v2) edge node[pos=0.5,right] {$e_5$} (l22);
\path (v3) edge node[pos=0.5,left] {$e_1$} (l11);
\end{tikzpicture}\caption{$G$}
	\end{subfigure}
	\begin{subfigure}[]{0.5 \textwidth}
		\centering
	\begin{tikzpicture}
 [scale=0.28, vertices/.style={draw, fill=black, circle, inner sep=0.5pt}]
\node[vertices, label=left:{$e_1$}] (a) at (-4,0) {};
\node[vertices, label=right:{$e_3$}] (b) at (4,0) {};
\node[vertices, label=above:{$e_4$}] (c) at (0,5) {};
\node[vertices, label=below:{$e_5$}] (d) at (0,-5) {};
\node[vertices, label=left:{$e_6$}] (e) at (-6,4.5) {};
\node[vertices, label=right:{$e_2$}] (f) at (12,0) {};

\foreach \to/\from in {a/b,a/c,a/d,b/c,b/d,a/e,c/f,d/f}
\draw [-] (\to)--(\from);
%\filldraw[fill=gray!60, draw=black] (2,3)--(9, 6.5)--(15,3)--cycle;
\filldraw[fill=gray!60, draw=black] (-4,0)--(4,0)--(0,-5)--cycle;
\filldraw[fill=yellow!50, draw=black] (-4,0)--(4,0)--(0,5)--cycle;
%\draw [dashed] (a)--(b);
\end{tikzpicture}\caption{$\mathrm{BD}^{(1,1,1,1,1,1,1)}(G)=M_1(G)$}\label{fig:BD of G}
	\end{subfigure}
	\caption{} \label{fig:example of BD}
\end{figure}

Bounded degree complexes were introduced by Reiner and Roberts in \cite{RR00} and further studied by Jonsson in \cite{Jon08}. For more on these complexes, interested reader is referred to \cite{Jon08,anurag2, Wach03}. 

 In \cite[Theorem 4.13]{MT08}, Marietti and Testa proved that the $1$- matching complexes of forests are homotopy equivalent to a wedge of spheres.  In \cite{Vega19}, Vega studied the homotopy type of $2$-matching complexes of caterpillar graphs (see \Cref{def:caterpillar graph}) and conjectured \cite[Conjecture 7.3]{Vega19} that the $k$-matching complex of caterpillar graphs are either contractible or homotopy equivalent to a wedge of spheres. The author, in \cite[Theorem 1.2]{anurag1}, proved this conjecture by showing that the bounded degree complexes of forests are homotopy equivalent to wedge of spheres. Recently, Matsushita \cite{mat20} showed that these complexes are shellable. Here, we strengthen his result by showing that the bounded degree complexes of forests are in fact vertex decomposable.

We first introduce a few notations. For $\vl:V(G)\longrightarrow \mathbb{Z}_{\geq 0}$ a labelling of $G$ and $v\in V(G)$, the induced labelling $\vl_{G,v}$ of graph $G-v$ is given by 
\begin{equation}\label{eq:labellingforv}
      \vl_{G,v}(u) = \vl(u)~~ \forall ~ u \in V(G-v)
  \end{equation}
 and for $e = (v,w)\in E(G)$, the labelling $\vl_{G,e}$ of graph $G-e$ is given by
 \begin{equation}\label{eq:labellingfore}
      \vl_{G,e}(u) = 
      \begin{cases}
        \vl(u), & \text{~if~} u \notin \{v,w\}, \\
        \vl(u)-1, & \text{~if~} u \in \{v,w\}.
      \end{cases}
  \end{equation}
  
\begin{theorem}\label{thm:mainbdcomplex}
Let $\F$ be a forest and $\vec{\lambda}:V(\F)\longrightarrow \mathbb{Z}_{\geq 0}$ be a labelling of its vertices. Then, $\bd^{\vl}(\F)$ is vertex decomposable.
\end{theorem}
\begin{proof}
  We prove this by induction on the number $n$ of edges in the forest. If $\F$ has only one edge then the result is clear. Let $\F$ has $n\geq 2$ edges and assume that all bounded degree complexes of forest with at most $n-1$ edges are vertex decomposable.
  
  If $\F$ has an isolated vertex $v$, then $\bd^{\vl}(\F)=\bd^{\vl_{\F,v}}(\F-v)$. Moreover, if $\vl(u)=0$ for a vertex $u \in V(\F)$ then also $\bd^{\vl}(\F)=\bd^{\vl_{\F,u}}(\F-u)$. Therefore, we can assume that $\F$ does not have any isolated vertex and $\vl(v)\neq 0$ for any $v \in V(\F)$. 
 
  If $\F$ has an isolated edge $e$ then the result is clear from \Cref{prop:join of vd} as $\bd^{\vl}(\F)$ is a cone over $\bd^{\vl_{\F,e}}(\F-e)$ with apex $e$. Otherwise, $\F$ will have a corner vertex, say $w$. Further, if deg$_{\F}(w)\leq \vl(w)$ then $\bd^{\vl}(\F)$ is a cone over $\bd^{\vl_{\F,\ell}}(\F-\ell)$ with apex $\ell$ for each leaf edge $\ell$ adjacent to $w$. In both cases the result follows from induction and \Cref{prop:join of vd}. 
  
 Now consider deg$_{\F}(w)> \vl(w)$. If $w$ is adjacent to an internal vertex $v$ then take $e=(w,v)$, otherwise choose $e=(w,v)$ for some $v \in N(w)$. 
  
  Observe that, 
  \begin{equation*}
  \begin{split}
      \lk(e, \bd^{\vl}(\F)) & = \bd^{\vl_{\F,e}}(\F-e),\mathrm{~and}\\
      \del(e, \bd^{\vl}(\F)) & = \bd^{\vl}(\F-e).
  \end{split}
  \end{equation*}
  Induction implies that both $\lk(e, \bd^{\vl}(\F))$ and $\del(e, \bd^{\vl}(\F))$ are vertex decomposable. Thus, it is now enough to show that $e=(w,v)$ is a shedding vertex of $\bd^{\vl}(\F)$.
  
  Let $\sigma$ be a facet of $\del(e, \bd^{\vl}(\F)) = \bd^{\vl}(\F-e)$ and $E_e=\{e^\prime\in E(\F-e) : w\in e^\prime\}$. Since $w$ is a corner vertex and deg$_{\F}(w)> \vl(w)$, $|E_e| \geq \vl(w)$ and all edges in $E_e$ are leaf edges of $\F$. Clearly, $E_e \subseteq E(\F-e)$ which implies that deg$_{(\F-e)[\sigma]}(w)=\sigma \cap E_e=\vl(w)$. Therefore, $\sigma \cup \{e\} \notin \bd^{\vl}(\F)$ implying that $\sigma$ is a facet of $\bd^{\vl}(\F)$.
\end{proof}

The technique used in the above proof can be applied to a bigger class of graphs. A subgraph $H$ of $G$ will be called \emph{cycle subgraph}, if $H\cong C_n$ for some $n\geq 3$. For $v\in V(G)$, $L_G(v)$ will denote the number of leaves adjacent to v in $G$.

\begin{proposition}\label{prop:bdcomplex}
Let $G$ be a graph and let $\vl$ be a labelling of $G$. If for each cycle subgraph $H$ of $G$ there is a vertex $v \in V(H)$ such that $L_G(v)\geq \vl(v)$, then $\bd^{\vl}(G)$ is vertex decomposable.
\end{proposition}
\begin{proof}
  Proceed by induction on the number $n$ of cycle subgraphs of $G$. If $n=0$, {\it i.e.}, $G$ has no cycle then the result follows from \Cref{thm:mainbdcomplex}.
  
  Now consider, $G$ is a graph with $n$ cycles and $H$ is a cycle subgraph of $G$. Let $v,w\in V(H)$ such that $L_G(v)\geq \vl(v)$ and $e=(v,w)\in E(H)$. It is easy to observe that, $\lk (e, \bd^{\vl} (G))=\bd^{\vl_{G,e}}(G-e)$ and the number of cycles in $G-e$ is less than $n$. Since $e$ is a non-leaf edge, $G-e$ and the labelling $\vl_{G,e}$ satisfy the hypothesis of \Cref{prop:bdcomplex}. Thus, by induction, $\lk (e, \bd^{\vl} (G))$ is vertex decomposable. Using similar arguments we get that $\del(e, \bd^{\vl} (G))=\bd^{\vl}(G-e)$ is also vertex decomposable. Moreover, for any facet $\sigma$ of $\del(e, \bd^{\vl} (G))$, $\deg_{G[\sigma]}(v) = \vl(v)$ (since $L_{G-e}(v)=L_G(v)\geq \vl(v)$) which implies that $\sigma$ is a facet of $\bd^{\vl} (G)$. This completes the proof of \Cref{prop:bdcomplex}.
\end{proof}

A \emph{fully whiskered graph} is a graph in which every non-leaf vertex is adjacent to at least one leaf vertex. The following is an immediate corollary of \Cref{prop:bdcomplex}.

\begin{corollary}
The $1$-matching complex of any fully whiskered graph is vertex decomposable. 
\end{corollary}

\section{The non-cover complex}\label{sec:noncover}
A subset $I \subseteq V(G)$ is called an \emph{independent set} of graph $G$ if the induced subgraph $G[I]$ does not have any edge. A subset $S \subseteq V(G)$ is called a \emph{cover} of $G$ if $V(G)\setminus S$ is an independent set of $G$. 

The \emph{independence complex} of a graph $G$, denoted as $\ind(G)$, is a simplicial complex whose simplices are all independent sets of $G$. The \emph{non-cover complex} of graph $G$, denoted  $\NC(G)$, is a simplicial complex whose simplices are non-covers of $G$. 

{\bf Example:}
\Cref{fig:example of NC} consists of graph $G_2(2,1)$ and $\NC(G_2(2,1))$. The complex $\NC(G_2(2,1))$ has $4$ factes, namely $\{v_1,v_3,e_4\}, \{v_2,v_3,v_5\}$, $\{v_2,v_4,v_5\}$ and $\{v_3,v_4,v_5\}$.

\begin{figure}[H]
	\begin{subfigure}[]{0.45 \textwidth}
		\centering
		\begin{tikzpicture}
 [scale=0.5, vertices/.style={draw, fill=black, circle, inner sep=1.0pt}]
        \node[vertices, label=below:{$v_1$}] (v1) at (0,0)  {};
		\node[vertices, label=below:{$v_2$}] (v2) at (5,0)  {};
		\node[vertices, label=above:{$v_3$}] (l11) at (-2.5,2.7)  {};
		\node[vertices, label=above:{$v_4$}] (l12) at (0,3.5)  {};
		\node[vertices, label=above:{$v_5$}] (l21) at (5.5,3.2)  {};
		
\foreach \to/\from in {v1/v2}
%\draw [-] (\to)--(\from);
\path (v1) edge node[pos=0.5,below] {} (v2);
\path (v1) edge node[pos=0.5,left] {} (l11);
\path (v1) edge node[pos=0.5,left] {} (l12);
\path (v2) edge node[pos=0.5,left] {} (l21);
\end{tikzpicture}\caption{$G_2(2,1)$}\label{fig:G221}
	\end{subfigure}
	\begin{subfigure}[]{0.45 \textwidth}
		\centering
	\begin{tikzpicture}
 [scale=0.32, every node/.style={draw=none}]
\node[label=left:{}, draw, fill=black, circle, inner sep=1.0pt] (a) at (0,0) {};
\node[label=left:{$v_3$},draw, fill=black, circle, inner sep=1.0pt] (b) at (-2,4) {};
\node[label=left:{$v_2$},draw, fill=black, circle, inner sep=1.0pt] (c) at (-4,-3) {};
\node[label=right:{$v_4$},draw, fill=black, circle, inner sep=1.0pt] (f) at (5,-1) {};
\node[label=left:{$v_5$},draw=none,inner sep=0.0pt] (g) at (-5.5,0.8) {};
\node[label=right:{$v_1$},draw, fill=black, circle, inner sep=1.0pt] (h) at (3.5,5.1) {};

\foreach \to/\from in {a/b,a/c,a/f,b/c,b/f,b/h,f/h}
\draw [-] (\to)--(\from);
%\filldraw[fill=gray!60, draw=black] (2,3)--(9, 6.5)--(15,3)--cycle;
\filldraw[fill=blue!30, draw=black] (0,0)--(-2,4)--(-4,-3)--cycle;
\filldraw[fill=yellow!50, draw=black] (0,0)--(-2,4)--(5,-1)--cycle;
\filldraw[fill=gray!50, draw=black] (0,0)--(5,-1)--(-4,-3)--cycle;
\filldraw[fill=gray!20, draw=black] (-2,4)--(3.5,5.1)--(5,-1)--cycle;
%\draw [dashed] (a)--(b);
\draw[dashed,thick, ->] (g) to (a);
\end{tikzpicture}\caption{$\NC(G_2(2,1))$}
	\end{subfigure}
	\caption{} \label{fig:example of NC}
\end{figure}

The \emph{(combinatorial) Alexander dual} $AD(K)$ of a simplicial complex $K$ is the simplicial complex 
$$AD(K)=\{\sigma \subseteq V(K) : V(K)\sm \sigma \notin K\}.$$
It is easy to see that  $\NC(G)$ is the Alexander dual of $\ind(G)$. Independence complexes have been studied extensively in last few decades and the homotopy type of these complex have been computed for various classes of graphs (for instance see \cite{ bar13, eng08, wood09}). Even though the reduced homology of $\NC(G)$ is related to that of $\ind(G)$ due to the Alexander duality theorem\footnote{{\bf Alexander duality theorem (\cite{stan82}):} Let $K$ be a simplicial complex and $V(K)\notin K$. Then for all $-1 \leq i \leq |V(K)|-2$, $\tilde{H}_i(AD(K))= \tilde{H}_{|V(K)|-i-3}(K)$. Here, $\tilde{H}_i(K)$ denotes the $i^{\text{th}}$ reduced homology group of $K$.}, the homotopy type of non-cover complexes remains mysterious. See \cite{MR14} for results related to the topology of the Alexander dual.

It is easy to observe that the facets of $\NC(G)$ are in one to one correspondence with the edges of $G$. In particular, any edge $(u,v)\in E(G)$ gives the unique facet $V(G)\setminus \{u,v\}$ and vice versa. 

\begin{lemma}\label{lemma:disjoint unionnoncover}
Let $G_1$ and $G_2$ are two disjoint connected graphs such that $|E(G_1)|, ~ |E(G_2)|\geq 1$. Then, $\NC(G_1\sqcup G_2)\simeq \Sigma (\NC(G_1)\ast \NC(G_2))$.
\end{lemma}
\begin{proof}
  Let $K_1=\Delta^{V(G_1)}\ast \NC(G_2)$ and $K_2=\NC(G_1)\ast \Delta^{V(G_2)}$. Since $|E(G_1)|, ~ |E(G_2)|\geq 1$, both $K_1$ and $K_2$ are contractible and $\NC(G_1\sqcup G_2)=K_1\bigcup K_2$. Thus, from \Cref{lemma:glueing}, we get that $\NC(G_1\sqcup G_2) \simeq \Sigma (K_1\cap K_2)$. Now the proof follows from the observation that $K_1\cap K_2 = \NC(G_1)\ast \NC(G_2)$.
\end{proof}

For a vertex $w\in V(G)$, let st$_w(G)$ denotes the graph on vertex set $N_G[w]$ and $E(\text{st}_w(G))=\{(v,w) : v \in N_G(w)\}$.

\begin{theorem}\label{thm:mainnoncover}
Let $v$ be a leaf vertex of graph $G$, $(v,w) \in E(G)$ and $\mathrm{deg}_G(w) >1$. Then, 
\begin{equation*}
    \begin{split}
        \NC(G) & \simeq \Sigma^{|N_G(w)|-1}\NC(G-N_G[w]).
    \end{split}
\end{equation*}
 \end{theorem}
 \begin{proof}
   Since $v$ is a leaf vertex, $\lk(w, \NC(G))$ is a cone with an apex $a$, hence contractible. Thus, from \Cref{eq:linkdel}, we have 
   \begin{equation*}
           \NC(G) \simeq \del(w,\NC(G))\bigvee \Sigma (\lk(w,\NC(G))) \simeq \del(w,\NC(G)).
   \end{equation*}
   
   \begin{claim}\label{claim:noncover}
   $\del(w,\NC(G)) = \del\big(w, \NC(\text{st}_w(G) \sqcup G-N_G[w])\big)$
   \end{claim}
   \begin{proof}[Proof of \Cref{claim:noncover}]
     Clearly, $\del\big(w, \NC(\text{st}_w(G) \sqcup G-N_G[w])\big) \subseteq \del(w,\NC(G))$. To show the other way inclusion, let $\sigma $ be a facet of $\del(w,\NC(G))$. We know that there exist $(a,b)\in E(G)$ such that $\si = V(G)\sm \{a,b,w\}$. If $(a,b)$ is an edge of $\text{st}_w(G)$ or of $G-N_G[w]$ then the result follows. Otherwise, without loss of generality assume that $a\in V(\text{st}_w(G))$. In this case $\sigma \subseteq V(G)\sm \{a,w\}$ which is a facet of $\del\big(w, \NC(\text{st}_w(G) \sqcup G-N_G[w])\big)$.
   \end{proof}
   From \Cref{lemma:disjoint unionnoncover}, it is easy to see that $\del\big(w, \NC(\text{st}_w(G) \sqcup G-N_G[w])\big) \simeq \NC(\text{st}_w(G) \sqcup G-N_G[w])$ (since $w\notin \NC(\text{st}_w(G)$).   Therefore, \Cref{thm:mainnoncover} follows from \Cref{lemma:disjoint unionnoncover}, \Cref{remark:join of space with spheres is suspension} and the observation that $\NC(\text{st}_w(G)= \partial(\Delta^{N_G(w)})$.
 \end{proof}
 
 Ehrenborg and Hetyei \cite{EH06} showed that, for any forest $\F$, the complex $\ind(\F)$ is either contractible or homotopy equivalent to a sphere. The following result is a direct consequence of \Cref{lemma:glueing} and \Cref{thm:mainnoncover}, it says that the Alexander dual of independence complexes of forests also have the same homotopy type.
 
 \begin{corollary}
 For any forest $\F$, the complex $\NC(\F)$ is either contractible or homotopy equivalent to a sphere.
 \end{corollary}

We now discuss the vertex decomposability of non-cover complexes. If $G$ contains an isolated vertex $v$ and at least one edge then $\NC(G)=C_v(\NC(G-v))$. Thus, \Cref{prop:join of vd} implies that $\NC(G)$ is vertex decomposable if and only if $\NC(G-v)$ is vertex decomposable. 

\begin{proposition}\label{prop:noncovershelling}
Let $G_1$ and $G_2$ are two disjoint graphs such that $|E(G_1)|, ~ |E(G_2)|\geq 1$. Then, $\NC(G_1\sqcup G_2)$ is not shellable hence not vertex decomposable.
\end{proposition}
\begin{proof}
  On contrary, assume that $\NC(G_1\sqcup G_2)$ is shellable with the shelling order $F_1,\dots, F_t$. Without loss of generality, assume that $F_1 = V(G_1\sqcup G_2)\sm \{a_1,b_1\}$ where $(a_1,b_1)\in E(G_1)$. Define $r=\text{min}\{ i \in [t] : V(G_2) \nsubseteq F_i\}$. Clearly, $1< r\leq t$. It is easy to see that, for each $1\leq j <r$, $|F_r\cap F_j| = |F_r|-2$. Thus, dim$\Big(\big(\bigcup\limits_{1\leq j <r}\Delta^{F_j}\big) \cap \Delta^{F_r}\Big)<\text{dim}(\Delta^{F_r})-1$ which contradicts the fact that $F_1,\dots,F_t$ is a shelling order.
\end{proof}

Recall that $v$ is a decomposing vertex of $\NC(G)$ if $v$ is a shedding vertex of $\NC(G)$, and both $\lk(v,\NC(G))$ and $\del(v,\NC(G))$ are vertex decomposable. Thus, $\NC(G)$ is vertex decomposable if and only if $\NC(G)$ has a decomposing vertex.

\begin{theorem}\label{theorem:mainnoncover2}
Let $\F$ be a forest without any isolated vertex. Then, the complex $\NC(\F)$ is vertex decomposable if and only if $\F$ is connected and has at most two internal vertices.
\end{theorem}
\begin{proof}
  Let $\F$ is connected. If $\F$ has at most one internal vertex then $\NC(\F)$ is the boundary of a simplex, hence vertex decomposable. Let $\F$ has exactly two internal vertices, {i.e.}, $\F$ is the caterpillar graph $G_2(m,n)$ for some $m,n\geq 1$. Let $v_1$ and $v_2$ are the internal vertices (cf. \Cref{fig:G221}). It is easy to see that,
  \begin{equation*}
      \begin{split}
          \lk(v_1,\NC(G_2(m,n))) & =\Delta^{[m]}\ast \NC(G_1(n)), \text{ and }\\
          \del(v_1,\NC(G_2(m,n)) & = \del \big(v_1, \NC(\text{st}_{v_1}(G_2(m,n)) \sqcup (G_2(m,n) - N_{G_2(m,n)}[v_1]))\big) \\
          & = del(v_1, \NC(G_1(m+1))\ast \Delta^{[n]}\\
          & = \partial (\Delta^{[m+1]}) \ast \Delta^{[n]}
      \end{split}
  \end{equation*}
  Thus, both $\lk(v_1,\NC(G_2(m,n)))$ and $\del(v_1,\NC(G_2(m,n)))$ are vertex decomposable from induction and \Cref{prop:join of vd}. To show that $v_1$ is a decomposing vertex, it is now enough to show that $v_1$ is a shedding vertex. Let $F=V(G_2(m,n))\sm \{a,b,v_1\}$ is a facet of $\del(v_1,\NC(G_2(m,n)))$ such that $(a,b)\in E(G_2(m,n))$. If $v_1\in \{a,b\}$ then clearly $F$ is facet of $\NC(G_2(m,n))$. Let $v_1\notin \{a,b\}$. In this case $v_2 \in \{a,b\}$ and $F\subsetneq V(\F)\sm \{v_1,v_2\}$ , which contradicts the fact that $F$ is a facet of $\del(v_1,\NC(G_2(m,n)))$. Hence, $\NC(G_2(m,n))$ is vertex decomposable.
  
  If $\F$ is not connected then $\NC(\F) $ is not vertex decomposable from \Cref{prop:noncovershelling}. Therefore, consider that $\F$ is a tree and has at least $3$ internal vertices. In this case, we show that any vertex of $\F$ is not a decomposing vertex of $\NC(\F)$. We prove this in three parts.
  \begin{enumerate}
      \item {\bf $v$ is a non-corner internal vertex}
      
      In this case, $\lk(v,\NC(\F))=\NC(\F-v)$ which is not shellable by \Cref{prop:noncovershelling} (since $\F-v$ is disjoint union of two graphs with at least one edge each) and hence not vertex decomposable. Therefore, any non-corner internal vertex is not a decomposing vertex. 
      
      \item {\bf $v$ is a corner vertex}
      
      Since $\F$ has at least $3$ internal vertices, we can choose a corner vertex $w\in V(\F)$ such that $v \notin N_{\F}[w]$. Let $x$ be a leaf vertex adjacent to $w$. Observe that, $F=V(\F)\setminus \{v,w,x\}$ is facet of $\del(v, \NC(\F))$ but $F$ is not a facet of $\NC(\F)$. Thus, $v$ is not a shedding vertex hence not a decomposing vertex.
      
      \item {\bf $v$ is a leaf vertex}
      
      The proof is similar to that of part (2). Choose vertices $w, x\in V(\F)$ such that $v \notin N_{\F}[w]$ and $x$ is a leaf vertex adjacent to $w$. Then, $V(\F)\setminus \{v,w,x\}$ is facet of $\del(v, \NC(\F))$ but not a facet of $\NC(\F)$ implying that $v$ is not a decomposing vertex.
  \end{enumerate}
  This completes the proof of \Cref{theorem:mainnoncover2}.
\end{proof}

% \section{Non $k$-matching complex}

\section{Complexes of directed trees}\label{subsec:directedtrees}

A multidigraph $G$ is a pair $(V,E)$ of finite sets $V$ and $E$ with two maps $s,t:E \longrightarrow V$. The sets $V$ and $E$ are called \emph{vertex set} and \emph{edge set} of $G$ respectively. An edge $e \in E$ will be denoted as $(s(e)\rightarrow t(e))$, here $s(e)$ is called the \emph{source} of $e$ and $t(e)$ is called the \emph{target} of $e$. If for every two distinct edges $e,e^\prime \in E$ either $s(e)\neq s(e^\prime)$ or $t(e)\neq t(e^\prime)$, then $G$ is called a \emph{digraph}. With every multigraph $G=(V,E)$, we can associate an (undirected) graph $G^{un}$ as follows: the vertex of $G^{un}$ is $V$ and two vertices $u$ and $v$ are adjacent in $G^{un}$ if and only if $(u\rightarrow v)\in E$ or $(v\rightarrow u)\in E$. A multidigraph $\F$ is called \emph{multidiforest} if its underlying graph $\F^{un}$ is a forest.

A \emph{directed cycle} of $G$ is a connected subgraph $C$ of $G$ such that each vertex of $C$ is the source of exactly one edge and target of exactly one edge. A \emph{directed forest} is a multidigraph $\F$ such that $\F$ does not contain any directed cycle and different edges of $\F$ have distinct targets.

\begin{definition}
\normalfont For a multidigraph $G=(V,E)$, the \emph{complex of directed trees} is a simplicial complex, denoted as $DT(G)$, whose simplices are the subsets $\sigma \subseteq E$
such that the induced subgraph $G[\sigma]$ is a directed forest.
\end{definition}

The study of complexes of directed trees of digraphs was initiated by Kozlov \cite{koz99}. Later, in \cite{MT08}, Marietti and Testa generalized these complexes for multidigraphs and showed that these complexes for multidiforsts are homotopy equivalent to a wedge of spheres. In \cite[Lemma 3.2]{joj13}, Joji{\'c}
 showed that these complexes are in fact vertex decomposable for those directed graphs $G$ such that $G^{un}$ is a forest. His proof was dependent on another result \cite[Theorem 1]{wood09} due to Woodroofe. Here we give an independent proof of the fact that complexes of directed trees of multidiforests are vertex decomposable. Observe that if $G^{un}$ is a forest for any directed graph $G$ then $G$ is a multidiforest. Hence, \Cref{thm:maindirectedtree} is an improvement on the previously known results.
 
 \begin{theorem}\label{thm:maindirectedtree}
 Let $\F$ be a multidiforest. Then the complex $DT(\F)$ is veretx decomposable. 
 \end{theorem}
 
 For a multidigraph $G=(V,E)$ and $e\in E$, let $G_{\downarrow e}$ denotes the multidigraph obtained from $G$ by first removing the edges with target $t(e)$, and then identifying the vertex $s(e)$ with the vertex $t(e)$ (see \Cref{fig:contractionofedge}). The multidigraph $G_{\downarrow e}$ was introduced by Marietti and Testa in \cite{MT08}. 
 
      \begin{figure}[H]
      	\begin{subfigure}[]{0.5\textwidth}
      \centering
     \begin{tikzpicture}
\node (a) at (0,0) {$u$};
\node (c) at (0,-2) {$v$};
\node (b) at (2,0) {$w$};
\node (d) at (2,-2) {$x$};
%\draw[-] (2.4,-1.7) -- (2.7,-1.7) -- (2.7,-1.4); 
\draw[->] (a)--(b) node[midway,above] {$e$};
\draw[->] (b)--(c) node[midway,above] {};
\draw[->] (a)--(c) node[midway,above right] {};
%\draw[right hook->] (c)--(d) node[midway,above right] {$\scriptstyle\iota$};
\draw[->] (b) edge (d) (c) edge (d);
\draw[->] (2.2,-1.8) -- (2.2,-0.2);

\end{tikzpicture}
\caption{multidigraph $G$}\label{diag:pushout1}
\end{subfigure}
\begin{subfigure}[]{0.5\textwidth}
      \centering
      \vspace{0.4cm}
      \begin{tikzpicture}
\node (c) at (0,-2) {$v$};
\node (b) at (1,0) {$w$};
\node (d) at (2,-2) {$x$};
%\draw[-] (2.4,-1.7) -- (2.7,-1.7) -- (2.7,-1.4); 
%\draw[->] (a)--(b) node[midway,above] {$e$};
\draw[->] (b)--(c) node[midway,above] {};
\draw[->] (0.76,-0.2)--(-0.04,-1.8) node[midway,above right] {};
%\draw[right hook->] (c)--(d) node[midway,above right] {$\scriptstyle\iota$};
\draw[->] (b) edge (d) (c) edge (d);
\end{tikzpicture}
\caption{multidigraph $G_{\downarrow e}$}\label{diag:pushout2}
\end{subfigure}
\caption{}\label{fig:contractionofedge}
\end{figure}

 Observe that,
 \begin{equation}\label{eq:multidigraphcontraction}
     \begin{split}
         \lk(e,DT(G)) & = DT(G_{\downarrow e}), \text{ and }\\
         \del(e,DT(G)) & = DT(G-e).
     \end{split}
 \end{equation}

\begin{proof}[Proof of \Cref{thm:maindirectedtree}]
  Proof is by induction on the number $n$ of edges of $\F$. The result is trivially true if $\F$ has only one edge. Let the result if true for any forest with at most $n-1$ edges and let $\F$ be a forest with $n$ edges. 
 
 Since both the multidigraphs $\F_{\downarrow e}$ and $\F -e$ have less number of edges, both the complexes $\lk(e,DT(\F))$ and $\del(e,DT(\F))$ are vertex decomposable from induction for each edge $e$ of $\F$. Therefore, to show that $DT(\F)$ is vertex decomposable, it is enough to find a shedding vertex of $DT(\F)$.
 
 If $e,f \in E(\F)$ such that $s(e)=s(f)$ and $t(e)=t(f)$ then $e$ is a shedding vertex. If not, then there exist a facet $\sigma \in \del(e,DT(\F))$ such that $\sigma \cup \{e\}\in DT(\F)$ but then $\sigma \cup \{f\}\in \del(e,DT(\F))$ which is a contradiction to the assumption that $\sigma$ is a facet of $\del(e,DT(G))$.
 
 Therefore, we can assume that $\F=(V,E)$ is a directed graph such that $\F^{un}$ is a forest. Let $u$ be a leaf vertex of $\F^{un}$ and $(u,v)\in E(\F^{un})$. Then there are following four possible cases:
 \begin{enumerate}
     \item $(v\rightarrow u) \in E$ and $(u\rightarrow v) \notin E$:
     
     In this case, observe that $DT(\F) = C_e( DT(\F-e))$ where $e=(v\rightarrow u)$. Therefore, $DT(\F)$ is vertex decomposable from induction and \Cref{prop:join of vd}. 
     
     \item $(v\rightarrow u) , (u\rightarrow v) \in E$:
     
     Let $e=(u\rightarrow v)$ and $f=(v\rightarrow u)$. Here we show that $e$ is a shedding veretx of $DT(\F)$. Let $\sigma$ be a facet of $\del(e,DT(\F))$. Since $e\notin \sigma$ and $\sigma$ is a facet, $f\in \sigma$ ( because $\sigma \cup \{f\} \in \del(e,DT(\F))$). Therefore, $\sigma \cup \{e\} \notin DT(\F)$ implying that $\sigma$ is a facet of $DT(\F)$.
     
     \item $(v\rightarrow u) \notin E$, $(u\rightarrow v) \in E$ and there is no $w (\neq u)\in V$ such that $(w\rightarrow v) \in E$:
     
     This case is similar to the case $(1)$. Here, $DT(\F) = C_f(DT(\F-f))$ with $f=(u\rightarrow v)$.
     
     \item $(v\rightarrow u) \notin E$, $(u\rightarrow v) \in E$ and there is an edge $(w\rightarrow v) \in E$ such that $w\neq u$:
     
     In this case, we show that the edge $g=(w\rightarrow v)$, where $w\neq u$, is a shedding vertex. Let $\sigma$ be a facet of $\del(g,DT(\F))$. If there exist $y \in V$ such that $y\neq w,u$ and $(y\rightarrow v) \in \sigma$ then $\sigma \cup \{g\} \notin DT(\F)$ implying that $\sigma$ is a facet of $DT(\F)$. Otherwise, $(u\rightarrow v) \in \sigma$ (because $\sigma$ is a facet of $\del(g,DT(\F))$ and $\sigma \cup \{(u\rightarrow v)\} \in \del(g,DT(\F))$) which again implies $\sigma \cup \{g\} \notin DT(\F)$.
     \end{enumerate}
     This completes the proof of \Cref{thm:maindirectedtree}.
\end{proof}
%\subsection{Complexes of rooted directed trees}

 \section*{Acknowledgements} The author would like to thank Priyavrat Deshpande for stimulating discussions and various suggestions on the early draft of this article.

\bibliographystyle{alpha}
%\bibliography{ref}

\end{document}